\documentclass[12pt,oneside,reqno]{amsart}
\usepackage{geometry}                % See geometry.pdf to learn the layout options. There are lots.
\usepackage{graphicx}
\usepackage{amssymb}
\usepackage{tikz}
\usetikzlibrary{calc}
\usepackage{bm}
\usepackage{mathtools}

\usepackage{amsmath,amsthm,amscd,amssymb}
\usepackage{latexsym}
\usepackage[colorlinks,citecolor=red,pagebackref,hypertexnames=false]{hyperref}
\numberwithin{equation}{section}
\theoremstyle{plain}
\newtheorem{theorem}{Theorem}[section]
\newtheorem{lemma}[theorem]{Lemma}
\newtheorem{corollary}[theorem]{Corollary}
\newtheorem{proposition}[theorem]{Proposition}

\theoremstyle{definition}
\newtheorem{definition}[theorem]{Definition}

\theoremstyle{remark}
\newtheorem{remark}[theorem]{Remark}

\newtheorem{case[theorem]}{Case}

\newcommand\numberthis{\addtocounter{equation}{1}\tag{\theequation}}

\date{\today}
\author{N. Chatzikonstantinou, A. Iosevich, S. Mkrtchyan and J. Pakianathan}
\address{}
\email{}

\thanks{}
\title{Rigidity, graphs and Hausdorff dimension}   
\begin{document}
%\maketitle

\begin{abstract} For a compact set $E \subset \mathbb R^d$ and a connected graph $G$ on $k+1$ vertices, we define a $G$-framework to be a collection of $k+1$ points in $E$ such that the distance between a pair of points is specified if the corresponding vertices of $G$ are connected by an edge. We regard two such frameworks as equivalent if the specified distances are the same. We show that in a suitable sense the set of equivalences of such frameworks naturally embeds in ${\mathbb R}^m$ where $m$ is the number of ``essential" edges of $G$. We prove that there exists a threshold $s_k<d$ such that if the Hausdorff dimension of $E$ is greater than $s_k$, then the $m$-dimensional Hausdorff measure of the set of equivalences of $G$-frameworks is positive. The proof relies on combinatorial, topological and analytic considerations. 
\end{abstract}

\thanks{The second and fourth listed authors were partially supported by the NSA Grant H98230-15-1-0319}
\maketitle
\tableofcontents

\section{Introduction} 

\vskip.125in 

The Falconer distance conjecture (\cite{Fal86}) says that if the Hausdorff dimension of a compact subset of ${\mathbb R}^d$, $d \ge 2$, is greater than $\frac{d}{2}$, then the Lebesgue measure of the set of distances determined by pairs of elements of the set is positive. The best current results are due to Wolff (\cite{W99}) in dimension $2$ and Erdogan (\cite{Erd05}) in ${\mathbb R}^d$, $d \ge 3$, who established the $\frac{d}{2}+\frac{1}{3}$ threshold. In the context of Ahlfors-David regular sets, the Falconer conjecture was established in the plane by Orponen (\cite{Orp15}). These results build partly on the previous work by Bourgain (\cite{B94}), Falconer (\cite{Fal86}) and Mattila (see \cite{M16} and the references contained therein). A related conjecture, also pioneered by Falconer (\cite{Fal86}) and studied extensively by Bourgain, says that if the Hausdorff dimension of $E$ is {\bf equal} to $\frac{d}{2}$, then the upper Minkowski dimension of $\Delta(E)$ is $1$. Bourgain proved that if $\dim_{{\mathcal H}}(E)=\frac{d}{2}$, $E \subset {\mathbb R}^d$, $d \ge 2$, then the upper Minkowski dimension of $\Delta(E)$ is $>\frac{1}{2}+\epsilon_d$ for some $\epsilon_d>0$.

The Falconer distance conjecture can be viewed as a problem about $2$-point frameworks. It is quite interesting to consider $(k+1)$-point frameworks with $k>1$. For example, one can consider triangles inside sufficiently large sets, properly interpreted. This problem has been extensively studied in a variety of contexts by Bennett, Bourgain, Chan, Furstenberg, Greenleaf, Katznelson, Laba, Pramanik, Weiss, Ziegler, the second and fourth named authors and others (see e.g. \cite{Bo86}, \cite{GILP13}, \cite{CLP14}, \cite{BHIPR13}, \cite{BIP12}, \cite{HI07}, \cite{ILiu15}, \cite{FKW90}, \cite{Z06}). In \cite{BIT16}, the authors considered chains, and in \cite{GIP14} necklaces were investigated. More general frameworks were studied in \cite{HLP16}. 

In this paper we show that in a suitable sense, a nontrivial dimensional threshold can be found for any finite point framework. What we mean by a finite point framework is a finite collection of points in a compact set $E$ of a given Hausdorff dimension, where some, but not all the pairwise distances are specified. We encode these frameworks in a rigorous way using combinatorial graphs. We then define a suitable notion of equivalence and embed the resulting equivalence classes in ${\mathbb R}^m$, where $m$ is the number of ``essential'' edges of the graph which encodes a given framework. We then prove that if the Hausdorff dimension of the ambient set $E$ is larger than a nontrivial threshold $s_k$, then the $m$-dimensional Hausdorff content of the set of equivalences is positive. The precise formulation can be found in Theorem \ref{maindude1} and Theorem \ref{maindude2} below. 

As the reader shall see below, a rigorous formulation of the Falconer type problem for finite point frameworks naturally leads one to the notions of rigidity and other interesting concepts that combine combinatorial, topological and analytic concepts. The resulting symbiosis makes possible results that were not accessible using the purely analytic methods that were employed in the cases of simplexes, chains and necklaces. 

\vskip.25in 

\section{Definitions and statements of results} 

We shall encode finite point frameworks using combinatorial graphs. Let $k\geq 1$ and let $K_{k+1}$ denote the complete graph with vertex set $\{1,\dots, k+1\}$ and edge set ordered lexicographically. Let $G_{k+1,m}$ be a subgraph of $K_{k+1}$ with $k+1$ vertices and $m$ edges inheriting the order. We define $ij\in G_{k+1,m}$ to mean that $i < j$ and $\{i,j\}$ is an edge of $G_{k+1,m}$, and when $ij\in G_{k+1,m}$ ranges over all the edges of $G_{k+1,m}$, it ascends in the order of the edge set. Let $|\cdot |$ denote the Euclidean distance, $\|\cdot \|_p$ the $L^p$ norm, $\dim (\cdot )$ denote the Hausdorff dimension and $\mathcal H^s(\cdot )$ denote the $s$-dimensional Hausdorff measure. Let $\mathcal L^m(\cdot )$ denote the Lebesgue measure of measurable subsets of $\mathbb R^m$. Let $A \lesssim B$ mean that for some constant $C>0$, $A \leq C\cdot B$, where $C$ is independent of $\epsilon$, $\delta$ and the summation index or integration variable (if used to bound the term). Moreover, $A \approx B$ means that $A \lesssim B$ and $B \lesssim A$.

\begin{definition}
A \emph{$(k+1)$-tuple $\bm x$ in $\mathbb R^d$} is a tuple
$$\bm x = (x^1, x^2, \dots, x^{k+1}), \ x^j \in \mathbb R^d\ .$$
\end{definition}

\begin{definition} A \emph{framework of $G_{k+1,m}$ in ${\mathbb R}^d$} is a pair $(G_{k+1,m}, \bm x)$, where $\bm x$ is a $(k+1)$-tuple in $\mathbb R^d$.\end{definition} 

A convenient way to specify distances is through the distance function which we now define.
\begin{definition} \label{defdistancefunction} Given a graph $G_{k+1,m}$ we define the \emph{distance function} $f_{G_{k+1,m}}(\bm x)$ on $\bm x = (x^1, \dots, x^{k+1}) \in \mathbb R^{d(k+1)}$ by
$$f_{G_{k+1,m}}(\bm x) = \left(|x^i - x^j|\right)_{ij\in G_{k+1,m}}\ .$$
We also define the \emph{distance-squared function} $F_{G_{k+1,m}}(\bm x)$ by
$$F_{G_{k+1,m}}(\bm x) = \left(|x^i - x^j|^2\right)_{ij\in G_{k+1,m}}\ .$$
\end{definition}

\begin{definition}[Graph Distances]\label{graphdists} The value $f_{G_{k+1,m}}(\bm x)$ is called the \emph{$G_{k+1,m}$-distance of $\bm x$}. When we restrict our domain to some set $\mathfrak X\subseteq \mathbb R^{d(k+1)}$, we call $f_{G_{k+1,m}}(\bm x)$ a \emph{$G_{k+1,m}$-distance on $\mathfrak X$} and we say that $\bm x$ is a realization of this distance in $\mathfrak X$. The set of $G_{k+1,m}$-distances on $\mathfrak X$ is $f_{G_{k+1,m}}(\mathfrak X)$ and we denote it by $\Delta(G_{k+1,m}, \mathfrak X)$. \end{definition}

\begin{remark}
Equivalence classes of frameworks (equivalent in the sense of the corresponding tuples having equal $f_{G_{k+1,m}}$-values) can be viewed as subsets of ${\mathbb R}^m$ since $m$ distances are specified in the sense of Definition \ref{graphdists}.  Given a graph, we ask whether there exists some $0<s_k<d$ such that any compact subset $E$ of ${\mathbb R}^d$, $d \ge 2$, of Hausdorff dimension larger than $s_k$ contains a positive $m$-dimensional measure (Hausdorff or Lebesgue, depending on the context) worth of equivalence classes of frameworks of the given graph, in other words, whether the set $\Delta(G_{k+1,m},E^{k+1})$ has positive (Hausdorff or Lebesgue) measure. Complete graphs in ${\mathbb R}^d$ with at most $d+1$ vertices were comprehensively studied in \cite{GILP13}. 
In fact as we shall see later, when $k\leq d$, the only interesting case is the complete graph. Thus in this paper we consider graphs with $k>d$ unless otherwise stated.
\end{remark}

\begin{remark} The distance set $\Delta(G_{k+1,m}, \mathfrak X)$ depends on the numbering of the vertices and the order of the edges. Whereas the order of the edges is superficial, inducing only a permutation in the components of the $G_{k+1,m}$-distances, the numbering of the vertices can significantly change the $G_{k+1,m}$-distance set. Consider $\mathfrak X = \{\bm x_0\}\times \mathbb R^d \times \cdots \times \mathbb R^d$ and a graph $G=G'\cup G''\cup\{e\}$ where $e$ is a bridge between $G'$ and $G''$. Then if we number the vertices of $G'$ followed by those of $G''$, we essentially capture $G''$-distances only, whereas if we reverse the numbering order of the vertices of $G$ we will capture $G'$-distances only. In the rest of this paper we take $\mathfrak X = E^{k+1}$ for some $E\subset \mathbb R^d$, so that the numbering of the vertices becomes superficial as well. In particular, the dimension of the $G_{k+1,m}$-distance set and its Hausdorff (or Lebesgue) measure are independent of the vertex numbering and edge order. \end{remark} 

We define the notion of independence for subsets of the edge set of $K_{k+1}$ and of maximal independence for subsets of the edge set of $G_{k+1,m}$. We define the set of generic tuples as the complement of the zero set of a certain polynomial. This notion is independent of the graph $G_{k+1,m}$, depending only on the dimension $d$ and the number of vertices $k+1$. 

Let us use the following notation for our matrices: If $a_{ij}$ is a matrix, $(i,j)\in I\times J$, then for $B\subseteq I, C\subseteq J$, we defined $a_{B,C}$ to be the submatrix $a_{ij}$ with $(i,j)\in B\times C$. 

\begin{definition}
We say that $\bm x \in \mathbb R^{d(k+1)}$ is a \emph{regular tuple of $F_{G_{k+1,m}}$} if $\operatorname{rank}DF_{G_{k+1,m}}$ attains its global maximum at $\bm x$. A framework $(G_{k+1,m}, \bm x)$ is a \emph{regular framework} if $\bm x$ is a regular tuple of $F_{G_{k+1,m}}$.
\end{definition}

\begin{definition}
A subset $H$ of the edge set of $K_{k+1}$ is called \emph{independent in $\mathbb R^d$ with respect to $\bm x_0 \in \mathbb R^{d(k+1)}$} if the row vectors of $DF_{K_{k+1}}(\bm x_0)$ corresponding to $H$ are linearly independent. We call $H$ \emph{independent in $\mathbb R^d$} if there exists some $\bm x_0$ so that $H$ is independent with respect to $\bm x_0$, and $\bm x_0$ is said to be a \emph{witness to the independence of $H$}. We also call $H$ a \emph{maximally independent (in $\mathbb R^d$) subset of edges of $G_{k+1,m}$} when it is independent and it is not contained in a larger independent edge set of $G_{k+1,m}$.
\end{definition}

\begin{definition}\label{defgenericpts}
For any nonempty independent in $\mathbb R^d$ set $H$ of edges of $K_{k+1}$ we define the polynomial $P_{H}(\bm x)$ to be the sum of squares of $|H|\times |H|$-minors of the submatrix of rows of $DF_{K_{k+1}}$ corresponding to edges of $H$.  Thus,
\begin{align*}
P_H(\bm x) = \sum_{\substack{A\subset\{1,\dots,d(k+1)\} \\ |A| = |H|}} \left|\det (DF_{K_{k+1}}(\bm x)_{H,A})\right|^2\ .
\end{align*}
Let $X_H$ denote the zero set of $P_H$.

We define the set of \emph{generic tuples of $\mathbb R^d$} to be the complement of the zero set $X$ of the polynomial $P(\bm x)$ defined by
\begin{align*}
P(\bm x) = \prod_{H\text{ independent}} P_H(\bm x)\ .
\end{align*}
We call $X$ the set of \emph{critical tuples of $\mathbb R^d$}.
\end{definition}
\begin{remark}
We have $X = \cup_H X_H$ where the union is taken over all the edge sets $H$ which are independent and the generic tuples are then equal to $\mathbb R^{d(k+1)}\setminus X$. Moreover, if a set $H$ of edges is independent then by Definition \ref{defgenericpts} it is generically independent, i.e. independent with respect to any generic $\bm x$. In fact, the set of generic tuples is precisely the set of tuples that simultaneously witness the independence of every independent edge set.
\end{remark}
\begin{remark}
The polynomial $P(\bm x)$ is nontrivial because every $P_H$ is nontrivial since there is at least one witness $\bm x_H$ for the independence $H$, which means that $P_H(\bm x_H) \not = 0$. Thus $X$ is a proper algebraic variety of dimension
\begin{align}
\dim X \leq d(k+1) - 1 \label{critdim}\ .
\end{align}
\end{remark}
\begin{remark}\label{genericimpliesregular}
It is immediate from the definitions that generic tuples are regular tuples. The other implication does not hold in general.
\end{remark}
\begin{definition}
A framework $(G_{k+1,m}, \bm x)$ is called \emph{generic in $\mathbb R^d$} if $\bm x$ is a generic tuple in $\mathbb R^d$ and it is called \emph{critical in $\mathbb R^d$} if $\bm x$ is a critical tuple in $\mathbb R^d$.
\end{definition}

Our main results concern the dimension of the set $\Delta(G_{k+1,m}, E^{k+1})$ and its Hausdorff (or Lebesgue) measure. An important role is played by properties of the graph $G_{k+1,m}$. In particular it is essential whether the graph is rigid or not. 

The key heuristic notion of this paper is that a graph $G_{k+1,m}$ is \emph{rigid in $\mathbb R^d$} if once the $m$ quantities $t_{ij}$ in
$$|x^i - x^j| = t_{ij}, \qquad ij\in G_{k+1,m}$$
are specified, the other distances $|x^i - x^j|$ for $ij\not\in G_{k+1,m}$ can only take finitely many values as the frameworks $(G_{k+1,m}, \bm x)$ vary over the set of generic frameworks. For technical reasons, we use a more precise and flexible notion of rigidity described below. A simple example that illustrates the technical obstacles one must contend with is the following. Consider a quadrilateral in the plane with side-lengths $1,1,1,3$. This configuration is perfectly rigid in the heuristic sense, but it is not {\it minimally infinitesimally rigid}, as the reader will see, roughly because the rigidity in this case is not stable under small perturbations. 

\vskip.125in 

We now turn to precise definition. 
 
\begin{definition}
An \emph{infinitesimal motion $\bm u = (u^1, \dots, u^{k+1})$ in $\mathbb R^d$ of $G_{k+1,m}$ at $\bm x$} is a $(k+1)$-tuple $\bm u$ of vectors $u^j\in \mathbb R^d$ such that
$$DF_{G_{k+1,m}}(\bm x)\cdot \bm u = 0\ .$$
The set of infinitesimal motions in $\mathbb R^d$ of $G_{k+1,m}$ at $\bm x$ is the kernel of $DF_{G_{k+1,m}}(\bm x)$.
Let us denote by $\mathcal V(G_{k+1,m}, \bm x)$ the set of infinitesimal motions in $\mathbb R^d$ of $G_{k+1,m}$ at $\bm x$. Let $\mathcal D(\bm x)$ be the set of infinitesimal motions in $\mathbb R^d$ of $K_{k+1}$ at $\bm x$.
\end{definition}
\begin{remark}
It is evident that $\mathcal D(\bm x) \subseteq \mathcal V(G_{k+1,m}, \bm x)$ since the system of equations $DF_{G_{k+1,m}}(\bm x)\cdot \bm u = 0$ is included in $DF_{K_{k+1}}(\bm x)\cdot \bm u = 0$.
\end{remark}
\begin{definition}\label{infrigidframework}
A framework $(G_{k+1,m}, \bm x)$ is called \emph{infinitesimally rigid in $\mathbb R^d$} when \linebreak
$\mathcal V(G_{k+1,m}, \bm x) = \mathcal D(\bm x)$.
\end{definition}

It is unnecessarily restrictive to require of a graph to have all its frameworks be infinitesimally rigid. We shall only require it of generic frameworks.
\begin{definition}\label{definfrigid}
A graph $G_{k+1,m}$ is called \emph{infinitesimally rigid in $\mathbb R^d$} if all its generic frameworks are infinitesimally rigid. It is called \emph{minimally infinitesimally rigid in $\mathbb R^d$} if it is infinitesimally rigid and no proper subgraph (on the same vertex set) is infinitesimally rigid.
\end{definition}

\subsection{Statements of results} 

\begin{theorem}\label{maindude1}
Let $G_{k+1,m}$ be a connected graph that is minimally infinitesimally rigid in $\mathbb R^d$, $d\geq 2$ and $E$ is a compact subset of $\mathbb R^d$ with $\dim E > d - \frac 1 {k+1}$.  Then 
$$\mathcal L^m(\Delta(G_{k+1,m},E^{k+1})) > 0\ .$$
\end{theorem}

\vskip.125in 

\begin{remark} \label{connectednotneeded} We shall in the proof of Theorem \ref{maindude1} that if $G_{k+1,m}$ is not connected, the proof naturally breaks into consideration of the connected components of the graph. \end{remark} 

\vskip.125in 

\begin{theorem}\label{maindude2}
Let $G_{k+1,m}$ be a graph without isolated vertices and let $E$ be a compact subset of $\mathbb R^d$, $d\geq 2$ with $\dim E > d - \frac 1 {k+1}$. Let $H$ be a maximally independent subset of edges of $G_{k+1,m}$. Then
$$\dim \Delta(G_{k+1,m}, E^{k+1}) = |H|$$
and
$$\mathcal H^{|H|}\left(\Delta(G_{k+1,m}, E^{k+1})\right) > 0\ .$$
\end{theorem}
\begin{remark}
We stated Theorem \ref{maindude2} using the Hausdorff measure instead of the Lebesgue measure because the edges not in $H$ do not produce any further dimensionality (see Corollary \ref{distdim} and Proposition \ref{notconnected}).
\end{remark}
\begin{remark}
It should be pointed out that for our results we only work out the case where $k>d$. Theorem \ref{maindude1} for $k \leq d$ was worked out in \cite{GILP13} and the better threshold $\dim E > \frac{dk+1}{k+1}$ was obtained. Theorem \ref{maindude2} follows as a consequence, since when $k\leq d$, the only minimally infinitesimally rigid graph is the complete graph on $k+1$ vertices and the independence condition of Theorem \ref{maindude2} is always satisfied (see Theorem \ref{completegraphindep}) so $H$ can be taken to be the edge set of $G_{k+1,m}$ and in that case Theorem \ref{maindude2} is a consequence of Theorem \ref{maindude1} by an application of Fubini's theorem.
\end{remark}
\begin{theorem}[Deforestation]\label{sidedude}
Let $G_{k+1,m}$, $m\geq 2$ be a graph without isolated vertices with a vertex $v$ of degree $1$ and let $G_1$ be the resulting graph when $v$ is removed from $G_{k+1,m}$. Iterate this process obtaining a sequence $G_1, \dots, G_n$, until $G_n$ has no more such vertices or when $G_n = K_2$. Then in using Theorem \ref{maindude1} or Theorem \ref{maindude2} with $G_{k+1,m}$, the dimensional threshold for $E$ obtained may be taken to be $$d - \frac {1}{k+1-n}\ .$$
\end{theorem}
\begin{remark}
Thus trees disjoint from the rest of the graph except for the root can be ignored by applying Theorem \ref{sidedude}.
\end{remark}

\vskip.125in 

We shall now see that our results are fairly sharp in the sense that the critical exponent must in general tend to $d$ as the number of vertices tends to infinity. 

\vskip.125in 

\begin{theorem} \label{necessaryexponent} Let $E$ be a compact subset of ${\Bbb R}^d$, $d \ge 2$ of Hausdorff dimension $s \in (0,d)$. Then the conclusion of Theorem \ref{maindude1} and Theorem \ref{maindude2} does not in general hold if $s<d-\frac{{d \choose 2}}{k}$. \end{theorem} 

\vskip.125in 

In dimension $d=2$, the difference between the exponents in Theorem \ref{maindude1}, Theorem \ref{maindude2} and Theorem \ref{necessaryexponent} is not very large, $2-\frac{1}{k+1}$ versus $2-\frac{1}{k}$. In higher dimensions the gap increases, but Theorem \ref{necessaryexponent} still shows that the correct critical exponent must in general tend to $d$ as the number of vertices tends to $\infty$. 

\vskip.25in 

\section{Graph distances of subsets of \texorpdfstring{$\mathbb R^d$}{Rd}} \label{graphsection}
\subsection{Introduction}
Our goal is to prove that
\begin{align}
\mathcal L^m(\Delta(G_{k+1,m}, E^{k+1})) > 0 \label{ourgoal}
\end{align}
for some dimensional threshold $s_k$ with $\dim E > s_k$. Here we may assume $G_{k+1,m}$ is connected, since we have
\begin{proposition}\label{notconnected}
If $G_1, \dots, G_n$ are the connected components of $G_{k+1,m}$ on $k_1, \dots, k_n$ vertices respectively, then for cartesian products $E^{k+1}$ where $E\subseteq \mathbb R^d$, we have
$$\Delta(G_{k+1,m}, E^{k+1}) = \Delta(G_1, E^{k_1})\times \cdots \times \Delta(G_n, E^{k_n})\ .$$
\end{proposition}
\begin{proof}
It is clear that (after reordering the vertices if necessary) $f_{G_{k+1,m}} = (f_{G_1}, \dots, f_{G_n})$ where $f_{G_{k+1,m}}$, $f_{G_j}$ are the corresponding distance functions of $G_{k+1,m}$, $G_j$. The result follows.
\end{proof}
Thus we only need to consider connected graphs, and requiring of $G_{k+1,m}$ to be connected in Theorem \ref{maindude1} is not an essential restriction. If (\ref{ourgoal}) does not hold, it may be the case that the dimension of the $G_{k+1,m}$-distance set is not full. Theorem \ref{maindude2} then provides its Hausdorff dimension and positivity of the Hausdorff measure. 

We define the notion of congruency for tuples and frameworks. 

\begin{definition} \label{defcongruencevectors}Let $\bm x$ be a $(k+1)$-tuple in $\mathbb R^d$ and define the set of tuples \emph{congruent} to $\bm x$ to be 
\begin{align}M_{\bm x}  = \{ (Tx^1, \dots, Tx^{k+1}) : T\in \operatorname{ISO}(d) \},\label{congruentpts} \end{align} where $\operatorname{ISO}(d)$ denotes the set of isometries of $\mathbb R^d$ to itself.\end{definition} 

\begin{definition} We say that the framework $(G,\bm x)$ is \emph{congruent} to the framework $(G',\bm x')$ if $G=G'$ and $\bm x$ is congruent to $\bm x'$ in the sense of Definition \ref{defcongruencevectors}. \end{definition} 

We now describe some examples.
\subsection{Examples of $\Delta(G_{k+1,m},E^{k+1})$}

In the case where $G_{k+1,m}$ is the complete graph $K_2$ on $2$ vertices and $E\subseteq \mathbb R^d$, we recover the distance set of $E$
$$\Delta(K_2, E^2) = \{|x - y| : x,y\in E\}\ .$$

Now consider the complete graph $K_4$. Let $d=2$ and $E = \mathbb R^2$. We will directly show that $\mathcal L^6(\Delta(K_4, \mathbb R^8)) = 0$. This is expected, as we will also show that $\dim \Delta(K_4, \mathbb R^8) = 5$ and $\mathcal H^5(\Delta(K_4, \mathbb R^8)) > 0$. Split the tuples $\bm x \in \mathbb R^8$ into two sets $A_1$ and $A_2$, $A_j \subset \mathbb R^{8}$, with $\bm x\in A_1$ iff $(x^1, x^2, x^3, x^4)$ are in convex position and $A_2$ the rest. We will work with $A_1$ but $A_2$ is treated similarly. 

\begin{center}
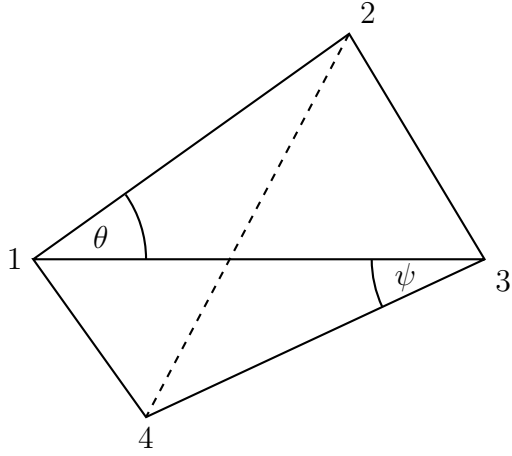
\begin{figure}[h]
\begin{tikzpicture}[scale=3]
\coordinate (A) at (0cm,0cm);
\draw (A) node[anchor=east] {$1$};
\coordinate (B) at (1.4cm,1cm);
\draw (B) node[anchor=south west] {$2$};
\coordinate (C) at (2cm,0cm);
\draw (C) node[anchor=north west] {$3$};
\coordinate (D) at (0.5cm,-0.7cm);
\draw (D) node[anchor=north] {$4$};
\draw[thick] (A) -- (B) -- (C) -- (D) -- (A) -- (C);
\draw[thick, dashed] (B) -- (D);
\draw[thick] (0.5cm,0cm) arc (0:35:0.5cm);
\draw[thick] (1.5cm,0cm) arc (180:205:0.5cm);
\draw (0.3cm, 0.1cm) node {$\theta$};
\draw (1.65cm, -0.08cm) node {$\psi$};
\end{tikzpicture}
\caption{A framework of $K_4$ (dashed edge for emphasis).}\label{figure}
\end{figure}
\end{center}
Let $t_{ij}$ denote the distance from the vertex $i$ to $j$. By using Euler's theorem for convex quadrilaterals we may obtain the following equation,
\begin{align}t^2_{24} = t^2_{23}+t^2_{14}-t^2_{13}+2t_{12}t_{34}\cos(\theta - \psi)\label{eulerquad}\ .\end{align}
Here $$\displaystyle \theta = \cos^{-1}\left( \frac{t^2_{12} + t^2_{13} - t^2_{23}}{2t^2_{12}t^2_{13}}\right), \psi = \cos^{-1}\left(\frac{t^2_{13} + t^2_{34} - t^2_{14}}{2t^2_{13}t^2_{34}}\right)\ .$$ 

Let $t_{24} = g(\bm {\tilde t})$ where $$ \bm {\tilde t} = (t_{12},t_{13},t_{14},t_{23},t_{34})\ .$$ Thus,
$$\mathcal L^6(\Delta(K_4, A_1)) \leq \int\int_{t_{24} = g(\bm {\tilde t})} d t_{24} d\bm {\tilde t} = 0\ .$$
This happens because equation (\ref{eulerquad}) makes us integrate over the zero-dimensional set (it is a singleton) $t_{24} \in \{ g(\bm {\tilde t}) \}$. Similarly we may obtain $\mathcal L^6(\Delta(K_4, A_2)) = 0$. Thus $K_4$ in $d=2$ cannot possibly give us a dimensional threshold since even $E=\mathbb R^2$ has a $K_4$-distance set of zero measure. This happened because the graph had too many edges. Not only $\Delta(K_4, \mathbb R^8)$ is a $\mathcal L^6$-null set, but in fact its dimension is less than $6$. By using Corollary \ref{distdim}, we find that its dimension is $5$. Then using Theorem \ref{maindude2} we see that it has positive $\mathcal H^5$-measure.

Now consider the following `double banana' graph $G_{8,18}$ on $\mathbb R^3$ (dashed edge for emphasis, but it is in the edge set of $G_{8,18}$),
\begin{center}
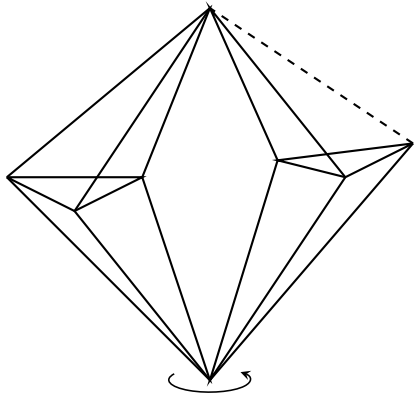
\begin{figure}[h]
\begin{tikzpicture}[scale=.45]
\draw[thick] (0,0) -- (4,0) -- (2,-1) -- (0,0);
\draw[thick] (0,0) -- (6,5) -- (4,0);
\draw[thick] (6,-6) -- (2,-1) -- (6,5);
\draw[thick] (0,0) -- (6,-6) -- (4,0); 
\draw[thick] (12,1) -- (8,.5) -- (10,0) -- (12,1);
\draw[thick,dashed] (12,1) -- (6,5);
\draw[thick] (8,.5) -- (6,5) -- (10,0);
\draw[thick] (12,1) -- (6,-6);
\draw[thick] (8,.5) -- (6,-6) -- (10,0);
\draw[->,>=stealth,semithick, yscale=0.3] (-2.66cm:20cm) arc[radius=1.2, start angle=150, end angle=400];
\end{tikzpicture}
\caption{A framework of the `double banana' $G_{8,18}$ graph.}\label{figure2}
\end{figure}
\end{center}
This is not an infinitesimally rigid graph. Each banana may be freely rotated about the line joining the banana ends without altering the edge lengths. Yet it may not be completed into a minimally infinitesimally rigid graph by adding edges because it contains redundant edges (the dashed one, for instance). The solid edges form a maximally independent set $H$ of edges of $G_{8,18}$, thus by an application of Theorem \ref{maindude2} we obtain $\dim \Delta(G_{8,18}, E^{8}) = 17$ and $\mathcal H^{17}(\Delta(G_{8,18}, E^{8})) > 0$ for $E\subseteq \mathbb R^3$ compact with $\dim E > 3 - \frac 1 {9}$.

\subsection{A sharp upper bound for the dimension of the distance set}
In this section we determine the Hausdorff dimension of $\Delta(G_{k+1,m}, \mathbb R^{d(k+1)})$. 

If $G_{k+1,m}$ is a minimally infinitesimally rigid graph in $\mathbb R^d$, then from Corollary \ref{noofedges}, it must have 
\begin{align}m = d(k+1) - \binom {d+1}2 \label{edgeno}\ .\end{align}
We may say that $G_{k+1,m}$ is minimally infinitesimally rigid in $\mathbb R^d$ when its edge set is independent and any edge added to $G_{k+1,m}$ turns the rows of $DF_{G_{k+1,m}}$ into a linearly dependent set of vectors, as the next proposition shows.
\begin{proposition}\label{maxindependent}
Let $G_{k+1,m}$ be a graph. Then the set of edges of $G_{k+1,m}$ is independent in $\mathbb R^d$ and may not be enlarged while retaining independence if and only if $G_{k+1,m}$ is minimally infinitesimally rigid.
\end{proposition}
\begin{proof}
If $G_{k+1,m}$ is minimally infinitesimally rigid in $\mathbb R^d$, then by definition for generic tuples $\bm x\in \mathbb R^{d(k+1)}$, the kernel of $DF_{G_{k+1,m}}(\bm x)$ has the smallest dimension possible (in view of Theorem \ref{genericgeneral} and the dimension of the rotation group). Thus the edge set can not be enlarged while retaining independence, since a larger independent set of edges would produce an even smaller kernel.

On the other hand, assume $G_{k+1,m}$ has an independent in $\mathbb R^d$ edge set that may not be enlarged while retaining independence.  We have just argued that $G_{k+1,m}$ cannot contain a minimally infinitesimally rigid proper subgraph. Assuming then that $G_{k+1,m}$ is not minimally infinitesimally rigid itself, for generic tuples $\bm x\in \mathbb R^{d(k+1)}$ we know that $\mathcal V(G_{k+1,m}, \bm x)$ properly contains $\mathcal D(\bm x)$. Thus there must be a set of edges $H \subset K_{k+1}$ disjoint from those of $G_{k+1,m}$ with
$$\mathcal V(H\cup G_{k+1,m}, \bm x) = \mathcal D(\bm x)\ .$$
But that implies $\operatorname{rank} D_{F_{H\cup G_{k+1,m}}}(\bm x) > \operatorname{rank} D_{F_{G_{k+1,m}}}(\bm x)$, a contradiction to the assumption that $G_{k+1,m}$ has an edge set that may not be enlarged while retaining independence.
\end{proof}

\begin{proposition}\label{constructcompletion}
If the edge set of $G_{k+1,m}$ is independent in $\mathbb R^d$, then a minimally infinitesimally rigid (in $\mathbb R^d$) graph $\overline G_{k+1,\overline m}$ containing $G_{k+1,m}$ exists.\end{proposition}
\begin{proof}
If $k \leq d$, we just complete $G_{k+1,m}$ to the complete graph $K_{k+1}$ since that is the only rigid graph on $k+1$ vertices in $\mathbb R^d$ (see Theorem \ref{completegraphindep}). 

Otherwise we pick $\bm x$ at random from a continuous distribution (say, the Gaussian distribution) on $\mathbb R^{d(k+1)}$. Since the set of critical frameworks is a proper algebraic variety, it has Lebesgue measure zero and we have almost certainly (that is, with probability $1$) picked a generic framework.

As long as the property of independence with respect to $\bm x$ is retained, we keep adding edges to $G_{k+1,m}$ until no more edges may be added. We then end up with a graph that is minimally infinitesimally rigid.
\end{proof}
\begin{remark}
The graph $\overline G_{k+1,\overline m}$ need not be unique. For instance, if $G_{k+1,m}$ is a tree, for large enough $k$ there are many different minimally infinitesimally rigid graphs it may complete to.
\end{remark}
\begin{theorem}\label{indepdim}
Let $G_{k+1,m}$ be a connected graph and $H$ a maximally independent in $\mathbb R^d$ subset of edges of $G_{k+1,m}$. Then
$$\dim \Delta(G_{k+1,m}, \mathbb R^{d(k+1)}) = |H|\ .$$
\end{theorem}
\begin{proof}
Let $W_{ij}$ denote the plane $\{\bm x : x^i = x^j\}$. The map $f_{G_{k+1,m}}$ is smooth away from $W = \bigcup_{ij\in G_{k+1,m}}W_{ij}$ and the rank of its total derivative does not exceed $|H|$ there. Thus, since $f_{G_{k+1,m}}$ has regular tuples away from $W$, we see that $f_{G_{k+1,m}}(\mathbb R^{d(k+1)} \setminus W)$ has dimension $|H|$. Now consider $f_{G_{k+1,m}}$ restricted to $W_{ij}$. There the function is smooth away from $W_{ij} \cap W_{kl}$ for $kl\in G_{k+1,m}$ with $kl\not = ij$. The rank of the derivative is less than or equal to $|H|$, so inductively $f_{G_{k+1,m}}(W_{ij})$ has dimension less than or equal to $|H|$.
\end{proof}
\begin{corollary}\label{distdim}
If $G_{k+1,m}$ is a connected graph that contains a minimally infinitesimally rigid (in $\mathbb R^d$) subgraph $G^*_{k+1,m^*}$, then
$$\dim \Delta(G_{k+1,m}, \mathbb R^{d(k+1)}) = m^*\ .$$
\end{corollary}
\subsection{Bounds on the number of noncongruent realizations}\label{realizationbounds}
Let $G_{k+1,m}$ be minimally infinitesimally rigid in $\mathbb R^d$ and let $\bm x\in \mathbb R^{d(k+1)}$.

We consider the set of preimages of $f_{G_{k+1,m}}(\bm x)$, $$N_{\bm x} = \{\bm y : f_{G_{k+1,m}}(\bm y) = f_{G_{k+1,m}}(\bm x)\}\ .$$ Define the equivalence relation $\bm y \sim \bm z$ by $\bm y \in M_{\bm z}$, where $M_{\bm z}$ is defined in (\ref{congruentpts}) to be the set of tuples congruent to $\bm z$. The set $N_{\bm x}$ is defined by the system of quadratic equations
$$|y^i - y^j|^2 = |x^i - x^j|^2, \quad ij \in G_{k+1,m}\ ,$$
and results on bounds of the Betti numbers of semi-algebraic varieties by Oleinik and Petrovskii, Thom and Milnor (see \cite{OPTS}, \cite{TSHV}, \cite{MILN}) allow us to conclude that $N_{\bm x}$, hence (since $\operatorname{ISO}(d)$ has two connected components), $N_{\bm x}/{\raise.17ex\hbox{$\scriptstyle\sim$}}$, has less than $C_{d,k} \cdot 2^{dk}$ connected components, for some $C_{d,k}>0$. For a better bound see \cite{BSEM}. In particular, when $\bm x$ is regular valued, we may conclude that there are at most $C_{d,k}\cdot 2^{dk}$ preimages of $f_{G_{k+1,m}}(\bm x)$ up to congruences by Proposition \ref{uniquereal}. If $\bm x$ is not regular valued, it is possible for noncongruent preimages to lie in the same connected component, but in the argument to follow we will avoid critical frameworks. For our purposes we only need the fact that if the critical tuples are removed, then the preimage set $N_{\bm x}$ is finite up to congruences, with the bound independent of $\bm x$.

\subsection{The proof of the dimensional threshold}
We prove Theorem \ref{maindude1}.
\begin{proof}
Let $G = G_{k+1,m}$ to ease subscript use. Fix $E\subset \mathbb R^d$ compact, and let $\mu = \mu_s$ be a Borel probability measure supported on $E$, with Frostman exponent $s$. Thus there exists some constant $C_\mu>0$ with $\mu(B(x,r)) \leq C_{\mu} r^s$ for all balls $B(x,r)$ of radius $r>0$, and we may choose $s < \dim E$ arbitrarily close to $\dim E$. (See \cite{W04}, Chapter 8 for the existence and properties of such measures). In particular, we may choose
\begin{align}
s > d - \frac 1 {k+1} \label{sdim}\ .
\end{align}
Let us first prove that the set of critical frameworks $X$ is a null set for $\mu^{k+1}$:
$$\mu^{k+1}(X) = 0\ .$$
Observe that $\mu^{k+1}$ is a Frostman measure of exponent $s(k+1)$. This follows easily from the fact that any ball $B(\bm x,r)$ is contained in a concentric cube $Q = Q_1\times \cdots \times Q_{k+1}$ of side $2r$, where each $Q_j$ in turn is contained in a ball $B(x^j, 2r)$. Since $\mu^{k+1}(Q) = \mu(Q_1)\cdots \mu(Q_{k+1}) \lesssim r^{s(k+1)}$ and $s(k+1) > \dim(X)$, which we may assume since the dimension of $E$ is big enough (by (\ref{critdim}) and (\ref{sdim})) and using the fact that  sets of positive measure of a Frostman measure have dimension greater or equal to the Frostman exponent (see Lemma \ref{frstman}), we conclude that $X$ is a null set for $\mu^{k+1}$.

Let $\delta > 0$ be such that $$\mu^{k+1}(E^{k+1}\setminus X_\delta) > 1/2$$ for $X_\delta$ the $\delta$-neighborhood of $X$ defined by $X_\delta = \{\bm y \in \mathbb R^{d(k+1)} : |\bm y - \bm x| < \delta\}$. Such a $\delta$ exists since $X$ is closed and $X = \bigcap_{\delta > 0} X_\delta$. We want to avoid getting close to $X$ because our named constants in the arguments to follow blow up near it. Let $A = E^{k+1}\setminus X_\delta$ and let $\nu(\bm t)$ be the pushforward of $\mu^{k+1}(\bm x)$ by $\left.f_G\right|_{A}$, that is, for any measurable function $g(\bm t)$ the integral $\int g d\nu$ is defined by
\begin{align*}
\int g(\bm t) d\nu(\bm t) = \int_{A} g\left((|x^i - x^j|)_{ij\in G}\right) d\mu(x^1)\cdots d\mu(x^{k+1})\ .
\end{align*}
From now on we shall write $d\mu^{k+1}(\bm x)$ for $d\mu(x^1)\cdots d\mu(x^{k+1})$.
We shall show that $\nu(\mathbb R^m) > 0$ and $\nu \in L^2(\mathbb R^m)$ implying $\mathcal L^m(\operatorname{supp} \nu) > 0$ which concludes the proof since $\operatorname{supp}\nu \subset \Delta(G, E^{k+1})$. For the first claim, 
\begin{align*}\nu(\mathbb R^m) = \mu^{k+1}(E^{k+1}\setminus X_\delta) > 1/2\ .\end{align*}
Now we will prove that $\nu \in L^2(\mathbb R^m)$. Let $\nu^\epsilon = \phi^\epsilon*\nu$, where $\phi^\epsilon(\bm t) =  \epsilon^{-m}\phi(\epsilon^{-1}\bm t)$ and $\phi\in C_c^\infty(\mathbb R^m)$ is a nonnegative radial function with $\int \phi = 1$, $\phi \leq 1$ and $\operatorname{supp}\phi \subset B(0, 2)$. Here we denote by $\chi$ the characteristic function of a set. We have,
\begin{align*}
\nu^{\epsilon}(\bm t) & = \int_A \epsilon^{-m} \phi\left(\epsilon^{-1}(f_G(\bm x) - \bm t)\right) d\mu^{k+1}(\bm x) \\
& \leq \int_A \epsilon^{-m} \chi\left\{ \big|f_G(\bm x) - \bm t\big| < 2\epsilon\right\} d\mu^{k+1}(\bm x) \numberthis\label{eqn1}\ .
\end{align*}
By Lemma \ref{basic1}, as $\epsilon \to 0$ we have $\nu^\epsilon \to \nu$ in the weak topology of the dual of $C_0(\mathbb R^m)$. We conclude that $\liminf \|\nu^\epsilon\|_2 \geq \|\nu\|_2$ from Lemma \ref{basic2}, and thus it suffices to bound $\liminf \|\nu^\epsilon\|_2$.
By an application of the triangle inequality on (\ref{eqn1}) we have,
\begin{align*}
\|\nu^\epsilon\|^2_2 & \leq \epsilon^{-2m} \int \int_A \int_A  \chi\left\{ \big|f_G(\bm x) - \bm t\big| < 2\epsilon\right\}\chi\left\{ \big|f_G(\bm y) - \bm t\big| < 2\epsilon\right\} d\mu^{k+1}(\bm x)d\mu^{k+1}(\bm y) d\bm t \\
& \leq \epsilon^{-2m} \int_{|\bm t| < 2\epsilon} d\bm t\cdot \int_A \int_A  \chi\left\{ \big|f_G(\bm x) - f_G(\bm y)\big| < 4\epsilon\right\}d\mu^{k+1}(\bm x)d\mu^{k+1}(\bm y) \\
& \lesssim \epsilon^{-m} \int_A \int_A  \chi\left\{ \big|f_G(\bm x) - f_G(\bm y)\big| < 4\epsilon\right\}d\mu^{k+1}(\bm x)d\mu^{k+1}(\bm y)\ .
\end{align*}
Consider $\bm y$ to be fixed now. Note that it is a regular tuple since it belongs to $A$. The condition that the images of $\bm x$ and $\bm y$ are $\epsilon$-close is giving us $\lesssim 2^{dk}$ open sets of $\mathbb R^{d(k+1)}$ where $\bm x$ may lie (as explained in Section \ref{realizationbounds}). Let $U_1, \dots, U_{l}$ be those open sets, and let $Z = \{\bm z_1, \dots, \bm z_l\}$ be such that each $\bm z_j\in U_j\cap A$ (potentially picking less than $l$ tuples, if some intersections are empty).  Denote by $O_d(\mathbb R)$ the group of rotations of $\mathbb R^d$. Cover the compact Riemannian manifold $O_d(\mathbb R)$ by $\epsilon$-balls  $T_1^\epsilon, \dots, T_{K(\epsilon)}^\epsilon$ of finite (uniformly in $\epsilon$) overlap with centers $g_1, \dots, g_{K(\epsilon)}$.

Then the set $\{\bm x\in A : |f_G(\bm x) - f_G(\bm y)| < \epsilon\}$ is a subset of the set
$$\bigcup_{\bm z\in Z} \bigcup_{k=1}^{K(\epsilon)} \{\bm x \in A : |(x^i - x^j) - g_k(z^i - z^j)| < c\epsilon,\quad  \forall ij\in G\}$$
for some $c>0$ that depends continuously on $\bm y$ (as Proposition \ref{uniquereal} shows, $f_G$ is biLipschitz in $U_1, \dots, U_l$, once congruences are identified). Since $A$ is a compact set, $c$ attains a maximum value, so pick such $c$ to lift the dependence on $\bm y$. 

Since $|Z| \lesssim 2^{dk}$, it follows that, 
\begin{align*}
\|\nu^\epsilon\|_2^2 & \lesssim \epsilon^{-m} \int_A \sum_{z\in Z} \sum_{k=1}^{K(\epsilon)}  \int_A \chi\{|(x^i - x^j) - g_{k}(z^i - z^j)| < c\epsilon,\  \forall ij \in G\} d\mu^{k+1}(\bm x)d\mu^{k+1}(\bm y) \\
& \lesssim 2^{dk} \epsilon^{-m} \sum_{k=1}^{K(\epsilon)} \int_A \int_A \chi\{|(x^i - x^j) - g_k(y^i - y^j)| < c\epsilon, \ \forall ij \in G\} d\mu^{k+1}(\bm x) d\mu^{k+1}(\bm y)\ .
\end{align*}

The volume of the $\epsilon$-balls of $O_d(\mathbb R)$ is $\approx \epsilon^{\dim O_d(\mathbb R)} = \epsilon^{d(d-1)/2}$. In what follows, we estimate the value of a function at a point by twice the average of that function around an $\epsilon$-ball and use the fact that they cover $O_d(\mathbb R)$ with finite overlap to obtain,
\begin{align*}
\|\nu^\epsilon\|_2^2 & \lesssim \epsilon^{-m} \sum_{k=1}^{K(\epsilon)} \frac{1}{\epsilon^{d(d-1)/2}} \int_{T_k^\epsilon}  \int_A \int_A \chi\{|(x^i - x^j) - g(y^i - y^j)| < c\epsilon,\  \forall ij \in G\} \\
& \hspace{10cm} d\mu^{k+1}(\bm x) d\mu^{k+1}(\bm y)dg \\
& \lesssim \epsilon^{-dk} \int_{O_d(\mathbb R)}\int_A\int_A \chi\{|(x^i - x^j) - g(y^i - y^j)| < c\epsilon,\ \forall ij \in G\} d\mu^{k+1}(\bm x) d\mu^{k+1}(\bm y)dg\ . \numberthis \label{ineqfullgraph}
\end{align*}
Here we used (\ref{edgeno}) to get $m+d(d-1)/2 = dk$. For $g\in O_d(\mathbb R)$, define $\nu_g$ by $$\int f(x)d\nu_g(x) = \int \int f(u-gv)d\mu(u)d\mu(v)\ .$$
Let $G'\subset G$ be a spanning tree. We continue (\ref{ineqfullgraph}) with
\begin{align}\|\nu^\epsilon\|_2^2  \lesssim \epsilon^{-dk} \int_{O_d(\mathbb R)}\int_A\int_A \chi\{|(x^i - x^j) - g(y^i - y^j)| < c\epsilon,\ \forall ij \in G'\} d\mu^{k+1}(\bm x) d\mu^{k+1}(\bm y)dg\ .\label{eqn2}\end{align}

Using Lemma \ref{natconvergence} (to be proved below) on (\ref{eqn2}) we obtain
\begin{align}
\|\nu^\epsilon\|_2^2  \lesssim \int \int \nu^{k+1}_g(x)dxdg\ .
\end{align} 
Theorem \ref{natmeasure} shows this integral to be finite for $s > d-\frac 1 {k+1}$, concluding the proof that $\nu\in L^2(\mathbb R^m)$ and thus showing that $\mathcal L^m(\Delta(G, E^{k+1})) > 0$. 
\end{proof}

We may now go a step further and prove Theorem \ref{maindude2}.
\begin{proof}
If $G_{k+1,m}$ is any connected graph, and $H$ is a maximally independent subset of the edge set of $G_{k+1,m}$, we may complete $H$ to a minimally infinitesimally rigid graph $\overline{H}_{k+1,\overline m}$ (see Proposition \ref{constructcompletion}). By using Theorem \ref{maindude1}, we obtain the nontrivial exponent $d - \frac 1 {k+1}$ for $\Delta(\overline H_{k+1,\overline m}, E^{k+1})$ to have positive Lebesgue measure. We project $\Delta(\overline H_{k+1,\overline m}, E^{k+1}) \to \Delta(H, E^{k+1})$ by $(t_{ij})_{ij\in \overline H_{k+1,\overline m}} \mapsto (t_{ij})_{ij\in H}$ to show that $\Delta(H, E^{k+1})$ has positive Lebesgue measure by Fubini. 

Now, projecting $\Delta(G_{k+1,m}, E^{k+1}) \to \Delta(H, E^{k+1})$ by $(t_{ij})_{ij\in G_{k+1,m}} \mapsto (t_{ij})_{ij\in H}$ shows that $\Delta(G_{k+1,m}, E^{k+1})$ has positive $\mathcal H^{|H|}$-measure, because the projection is Lipschitz. Lastly, Theorem \ref{indepdim} shows that $\dim \Delta(G_{k+1,m}, E^{k+1}) = |H|$. 

Moreover, if $G_{k+1,m}$ has connected components $G_1, \dots G_n$, we will obtain a maximally independent subset $H = H_1 \cup \cdots \cup H_n$ of $G_{k+1,m}$ where each $H_j$ is a maximally independent subset of $G_j$ for $j=1,\dots,n$. Using Proposition \ref{notconnected} and what we just argued for connected graphs, we again obtain a positive $|H|$-Hausdorff measures worth of distances.
\end{proof}

We now prove Theorem \ref{sidedude}.
\begin{proof}
As before let $G = G_{k+1,m}$ to avoid notational clutter. Let $\sigma_t$ denote the surface measure of the sphere $tS^{d-1} \subset \mathbb R^d$ of radius $t>0$ centered at $0$. Let $\phi^\epsilon(x) =  \epsilon^{-d}\phi(\epsilon^{-1}x)$ where $\phi\in C_c^\infty(\mathbb R^d)$ is a nonnegative radial function with $\phi = 1$ on $B(0,1)$, $\phi \leq 1$ and $\operatorname{supp}\phi \subset B(0, 2)$. Let $\sigma_t^\epsilon = \phi^\epsilon * \sigma_t$. We note that
\begin{align}
c\cdot \chi_{\{y: \left||y|-t\right|<\epsilon\}}(x)\leq \epsilon \sigma^\epsilon_t(x) \leq C\cdot \chi_{\{y: \left||y|-t\right|<2\epsilon\}}(x)\label{sidedudeeq1}\end{align}
for some constants $c>0$ and $C>0$ depending on $\phi$ only.
Without loss of generality assume $v = k+1$ and that $\{k,k+1\}$ is an edge of $G$. Let $\bm t = (t_{ij})_{ij\in G}$, $t_{ij}\in \mathbb (0,+\infty)$, and consider the function (on the domain of $\bm t$ just mentioned),
$$\Lambda^\epsilon_{G,\mu}(\bm t) = \int \prod_{ij\in G} \sigma^\epsilon_{t_{ij}}(x^i - x^j) d\mu(x^1)\cdots d\mu(x^{k+1})\ .$$
Using (\ref{sidedudeeq1}) we see that obtaining a bound $\|\Lambda^\epsilon_{G,\mu}\|_2 \leq M$ with $M$ independent of $\epsilon$ is equivalent to obtaining a bound for $\|\nu^\epsilon\|_2$ independent of $\epsilon$, in particular showing that $\nu\in L^2(\mathbb R^m)$.
Write then
\begin{alignat*}{2}
\Lambda^\epsilon_\mu(\bm t) & = \int \prod_{ij\in G, ij \not= k(k+1)} & \sigma^{\epsilon}_{t_{ij}}(x^i - x^j)d\mu(x^1)\cdots d\mu(x^{k-1}) \hspace{20cm} \\
&& \cdot \int \sigma^{\epsilon}_{t_{k(k+1)}}(x^k - x^{k+1}) d\mu(x^k)d\mu(x^{k+1}) \hspace{20cm} \\
& = \mathrlap{\int \left(\prod_{ij\in G, ij \not= k(k+1)} \sigma^{\epsilon}_{t_{ij}}(x^i - x^j)d\mu(x^1)\cdots d\mu(x^{k-1})\right) (\sigma^\epsilon_{t_{k(k+1)}}*\mu)(x^k) d\mu(x^k)}
\end{alignat*}
From Lemma \ref{sidedudelemma} we know that $\|\sigma^\epsilon_t*\mu\|_{L^1(\mu)} \lesssim 1$ and so using Chebyshev's inequality we may obtain a compact subset $E' \subset E$ with $\mu(E') > 0$ and $\|\sigma^\epsilon_t*\mu\|_{L^\infty(E',\mu)} \lesssim 1$. Denoting by $\mu'$ the restriction of $\mu$ to $E'$, we obtain a Frostman measure of the same exponent. Denote by $G'$ the graph $G$ with the vertex $v$ removed. Thus we have now
\begin{align*}
\Lambda^\epsilon_{G,\mu'}(\bm t) & \lesssim \int\prod_{ij\in G'} \sigma^\epsilon_{t_{ij}}(x^i - x^j)d\mu'(x^1)\cdots d\mu'(x^{k+1}) \\
& = \Lambda^\epsilon_{G',\mu'}(\bm t)
\end{align*}
The rest of the proof proceeds as in Theorem \ref{maindude1}, with $\mu'$ in place of $\mu$ and $G'$ in place of $G$.
\end{proof}
\subsection{The natural measure \texorpdfstring{$\nu_g$}{nu g} on \texorpdfstring{$E-gE$}{E-gE}}
We denote by $S^{d-1}$ the $(d-1)$-dimensional unit sphere centered at $0$ in $\mathbb R^d$. We will need the following result.
\begin{theorem}[Wolff-Erdo\u gan Theorem]\label{wolfferdogan} Let $\mu$ be a compactly supported Borel measure in $\mathbb R^d$. Then, for $s\geq d/2$ and $\epsilon >0$,
$$\int_{S^{d-1}} |\hat\mu(t\omega)|^2 d\omega \leq C_\epsilon I_s(\mu) t^{\epsilon-\gamma(s,d)}\ ,$$
with $\gamma(s,d) = (d+2s-2)/4$ if $d/2 \leq s \leq (d+2)/2$ and $\gamma(s,d) = s-1$ for $s\geq (d+2)/2$ where $I_s(\mu)$ is the $s$-energy of $\mu$, $I_s(\mu) = \iint |x-y|^{-s} d\mu(x)d\mu(y)$.
\end{theorem}
For $s \leq (d+2)/2$, see Wolff \cite{W99} for $d=2$, and Erdo\u gan \cite{Erd05} for $d\geq 3$. For the case $s \geq (d+2)/2$, see  Sj\"olin \cite{SESF}). See \cite{W04} Chapter 8 for the definition and relevant properties of $s$-energy, we are only interested in the fact that it will be a finite number.

\begin{theorem}[Natural measure on \texorpdfstring{$E-gE$}{E-gE}]\label{natmeasure}
Let $k \geq 2$ and let $E\subset \mathbb R^d$, $d\geq 2$ be a compact set with $\dim E > \frac d 2$. Let $\mu$ be a Borel probability measure on $E$ of Frostman exponent $s < \dim E$ with $s$ satisfying 
\[ \begin{cases}
	s > \frac{d(4k-1) + 2}{4k+2} & \text{ for } \qquad \frac d 2 < \dim E \leq \frac {d+2}2 \\
	s > \frac{4kd - 1}{4k + 1} & \text{ for }  \qquad \frac {d+2}2 < \dim E\ .
\end{cases}
\]
Let $g$ be a rotation, $g\in O_d(\mathbb R)$ and define the measure $\nu_g$ by
\begin{align}\int f d\nu_g = \int \int f(u - gv) d\mu(u) d\mu(v)\label{natmeasdef}\ .\end{align}
Then the integral 
\begin{align*}
\int \int \nu^{k+1}_g(x)dxdg
\end{align*}
is a finite quantity and in particular $\nu_g$ is absolutely continuous for $g$-a.e.
\end{theorem}
\begin{remark}\label{remarktable}
The threshold for $s$ when $\frac d 2 < \dim E \leq \dim \frac d 2 + 1$ is not useful unless $d=2$ or $d=3, k=1,2$ and $k=1$, since in that case $$\frac{d(4k-1)+2}{4k+2} \geq \frac {d+2}2 \geq \dim E\ .$$ In particular, below is a table for the readers convenience.

\renewcommand{\arraystretch}{2}

\begin{center}
\begin{tabular}{ |l|l| }
  \hline
  \multicolumn{2}{|c|}{Exponents for $s$ when $\displaystyle \frac d2 < \dim E \leq \frac {d+2} 2$} \\
  \hline
  $d=2$ & $\displaystyle s > \frac {4k}{2k+1}$ \\
  \hline  
  $d=3, k=1,2$ & $\displaystyle s > \frac {12k - 1}{4k + 2}$ \\
  \hline
  $d>3, k=1$ & $\displaystyle s > \frac {d}{2} + \frac 1 3$ \\
  \hline
\end{tabular}
\end{center}
If the values of $d,k$ are not listed on this table then we have 
$$ \dim E > s > \frac{4kd - 1}{4k+1}.$$
\end{remark}
\begin{proof}
Let $\psi\geq 0$ be a smooth radial function supported in $\{\xi \in \mathbb R^d : 1/2 \leq |\xi| \leq 4\}$, identically equal to $1$ in $\{\xi \in \mathbb R^d : 1 \leq |\xi| \leq 2\}$ with $\sqrt \psi$ also smooth. Let $\psi_j(\xi) = \psi(2^{-j}\xi)$. Moreover, we require $\sum_{j=-\infty}^{+\infty} \psi_j(\xi) = c$, for a suitable constant $c>0$, for all $\xi\not = 0$ (see \cite{W04} \S 7 for existence of such $\psi$). Let $\nu_{g,j}$ denote the $j$-th Littlewood-Paley piece of $\nu_g$ defined by $\hat \nu_{g,j}(\xi) = \hat \nu_g(\xi)\psi_j(\xi)$. Since $\nu_g$ is a finite measure, in bounding the pieces, we may assume that $j$ is bounded from below. Using the Littlewood-Paley decomposition of $\nu_g$, we may write $\int \nu^{k+1}_g(x)dx$ as
$$\int \sum_{j_1,\dots,j_{k+1}} \nu_{g,j_1}(x)\cdots \nu_{g,j_{k+1}}(x) dx\ .$$
This is bounded above by
$$(k+1)!\int \sum_{j_1\geq j_2 \geq \dots \geq j_{k+1}} \nu_{g,j_1}(x)\cdots \nu_{g,j_{k+1}}(x) dx\ .$$
Now using Plancherel, we see that since $\hat \nu_{g,j_2}*\dots *\hat \nu_{g,j_{k+1}}$ is supported on scale $2^{j_2} + \cdots + 2^{j_{k+1}} \leq 2^{j_2+1}$ while $\hat{ \nu}_{g,j_1}$ is supported on an annulus of scale $2^{j_1}$, the sum vanishes if $j_1 - j_2 >2$ for $j_2$ large. Thus it suffices to consider the case $j_1=j_2$ (the case $j_1 = j_2+1$ is similarly treated) and to look at the sum
\begin{align}
\sum_{j_1=j_2\geq j_3 \geq \cdots \geq j_{k+1}} \int \nu_{g,j_1}^2(x) \nu_{g,j_3}(x)\cdots \nu_{g,j_{k+1}}(x)dx\ . \label{orderedindices}
\end{align}
From the definition of $\nu_{g,j}$ it follows that $\nu_{g,j} = \mu_j(-\cdot) * \mu_j(g\cdot)$, where $\hat \mu_j = \hat \mu \sqrt{\psi_j}$.

By Young's inequality,
$$\|\nu_{g,j}\|_{\infty} \leq \|\mu_j\|_1\cdot \|\mu_j\|_\infty\ .$$
Trivially $\|\mu_j\|_1 \leq 1$ since $\mu$ is a probability measure. Also
$$|\mu_j(x)| = 2^{dj}|\mu * \widehat {\sqrt\psi}(-2^jx)| \leq C_N 2^{dj} \int (1 + 2^j|x - y|)^{-N}d\mu(y) \leq C'_N 2^{j(d-s)}$$ for any $N>1$
since $\mu$ is a Frostman measure on $E$. Using this estimate on the terms corresponding to the indices $j_3, \dots, j_{k+1}$ we can bound (\ref{orderedindices}) by a constant multiple of
$$\int \sum_{j} \left(\sum_{j\geq j_3\geq \dots\geq j_{k+1}} 2^{(j_3 + \cdots + j_{k+1})(d-s)}\right) \nu_{g,j}^2(x)dx \lesssim \int \sum_{j} 2^{j(k-1)(d-s)} \nu_{g,j}^2(x)dx\ .$$
It follows that
$$\int\int \nu^{k+1}_g(x)dxdg \lesssim \sum_{j} 2^{j(k-1)(d-s)}\cdot \int \int \nu^2_{g,j}(x)dxdg\ .$$
We will show that $\int \int \nu^2_{g,j}(x) dxdg \lesssim 2^{j(d-s)}2^{-j\gamma(s,d)}$ where the quantity $\gamma(s,d)$ is defined in Theorem \ref{wolfferdogan}, which completes the proof.

Since we have $\hat \nu_{g,j} = \hat \mu_j(\xi) \hat \mu_j (g\xi)$, via Plancherel we obtain
\begin{align*}
\int\int \nu^2_{g,j}(x)dxdg & = \int \int |\hat \mu_j(\xi)|^2 |\hat \mu_j(g\xi)|^2 d\xi dg \\
& = \int_0^\infty \int_{S^{d-1}} |\hat \mu_j(t\omega)|^2 \left(\int |\hat \mu_j(g t\omega)|^2dg \right) t^{d-1} d\omega dt\ .
\end{align*}
Since $O_d(\mathbb R)$ acts transitively on the sphere, the quantity in the parentheses is constant in $\omega$, and in particular it is a constant multiple of $\int |\hat \mu_j(t\omega')|^2 d\omega'$. Thus we have that
\begin{align*}\int\int \nu^2_{g,j}(x)dxdg & = C \int \left(\int_{S^{d-1}} |\hat\mu_j(t\omega)|^2d\omega\right)^2 t^{d-1} dt \\
& = C \int \left(\int_{S^{d-1}} |\hat\mu(t\omega)|^2 \psi(2^{-j}t\omega) d\omega\right)^2 t^{d-1} dt \\
& = C' \int_{2^{j-1}}^{2^{j+2}} \left(\int_{S^{d-1}} |\hat\mu(t\omega)|^2 d\omega\right)^2 t^{d-1} dt\ .
\end{align*}
Since we are summing over $j$ and the intervals $[2^{j-1}, 2^{j+2}]$ have finite overlap with each other, we may as well bound $\sum \int\int \nu^2_{g,j}(x)dxdg$ by a constant multiple of
$$\int_{2^{j}}^{2^{j+1}} \left(\int_{S^{d-1}} |\hat\mu(t\omega)|^2 d\omega\right)^2 t^{d-1} dt\ .$$

The proof is finished by using Theorem \ref{wolfferdogan}, showing that $\int\int \nu^2_g(x)dxdg < \infty$.
\end{proof}

\vskip.125in 

The proof is now complete up the proof of Lemma \ref{natconvergence} and the geometric results in Section 4. We prove the lemma below. The geometric results are established in Section \ref{geometrychapter}.

\begin{lemma}\label{natconvergence}
Let $G'_{k+1,m}$ be a tree. Then for small enough $\epsilon$ 
\begin{align}\label{myeq2}
\epsilon^{-dk} \int_{O_d(\mathbb R)}\int_A\int_A \chi\{|(x^i - x^j) - g(y^i - y^j)| < c\epsilon : ij \in G'_{k+1,m}\} d\mu^{k+1}(\bm x) d\mu^{k+1}(\bm y)dg
\end{align}
is bounded by a constant multiple of
$$\int\int \nu_g^{k+1}(x)dxdg\ .$$
\end{lemma}
\begin{proof}
First we bound (\ref{myeq2}) by
\begin{align}\label{myeq3}\epsilon^{-dk}\int \int\int \chi\{|(x^1 - x^j) - g(y^1 - y^j)| < (k+1)c\epsilon : j>1\} d\mu^{k+1}(\bm x) d\mu^{k+1}(\bm y)dg\ .\end{align}
This is accomplished as follows: Fix $ij\in G'_{k+1,m}$ and let $(x^1, x^{l_2}, \dots, x^{l_p}, x^i, x^j)$ be a path of length $p+1$, from $x^1$ to $x^j$ in $G'_{k+1,m}$. Set $l_1 = 1$ and $l_{p+1} = i$, $l_{p+2} = j$. Using the triangle inequality we see that the set 
$$\{(\bm x, \bm y) : |(x^i - x^j) - g(y^i - y^j)| < c\epsilon\}$$
is contained in the set
$$\{(\bm x, \bm y) : \sum_{f=1}^{p+1} |(x^{l_f} - x^{l_{f+1}}) - g(y^{l_f} - y^{l_{f+1}})| < (p+1)c\epsilon\}$$
which is contained in
$$\{(\bm x, \bm y) : |(x^1 - x^j) - g(y^1 - y^j)| < (p+1)c\epsilon\}\ .$$
Since $k \geq p$ we conclude that (\ref{myeq2}) is bounded by (\ref{myeq3}). Using now the natural measure $\nu_g$ we write (\ref{myeq3}) as
$$\epsilon^{-dk} \int\int\cdots \int \chi\{ | z^1 - z^j | < c\epsilon : j > 1\} d\nu_g(z^1)\cdots d\nu_g(z^{k+1}) dg\ .$$
Now it is obvious that taking $\epsilon \to 0$ and using the absolute continuity of $\nu_g$ and the dominated convergence theorem, we may bound (\ref{myeq3}) by a constant multiple of
$$\int\int \nu_g^{k+1}(z)dzdg$$
finishing the proof.
\end{proof}

\section{Geometric results}
\label{geometrychapter} 

For $d\geq 2$ and each $1 \leq q \leq d$ we show that the edge set of $K_{q+1}$ is independent in $\mathbb R^d$. We show that infinitesimal rigidity of a fixed graph $G_{k+1,m}$ is a generic property, either holding for all generic frameworks, or none of them. We count the number of edges a minimally infinitesimally rigid graph must have. The behavior of the distance function near regular tuples is investigated. Our approach follows closely that of \cite{CRGSS}. See also \cite{RFFR} for motivation and examples.
\subsection{Generic Frameworks}
In this section we prove various results for generic frameworks in $\mathbb R^d$. Some of the statements are for regular frameworks, but as noted in Remark \ref{genericimpliesregular}, generic frameworks are regular.

The following lemma, while technically obvious, serves to remind the reader of the form that $DF_{G_{k+1,m}}$ takes, which will be useful in subsequent proofs in this section.
\begin{lemma}\label{infmotions}
Fix $d\geq 2$. We have $DF_{G_{k+1,m}}(\bm x) \cdot \bm u = 0$ if and only if $\bm u$ is a solution to the following system of $m$ equations in $d(k+1)$ variables:
\begin{align}(x^i - x^j)\cdot(u^i - u^j) = 0, \text{ for } ij\in G_{k+1,m}\ .\label{infrigcondition}\end{align}
\end{lemma}
\begin{proof}
Since $F_{G_{k+1,m}}$ is a function $\mathbb R^{d(k+1)}\to \mathbb R^m$, $DF_{G_{k+1,m}}$ is a $m\times d(k+1)$ matrix with rows corresponding to edges $ij\in G_{k+1,m}$ and columns corresponding to the scalar components of $\bm x$. The $ij$-th row is equal to the following vector
$$(0, \dots, 0, \overbrace{2(x^i - x^j)}^{d(i-1)+1\text{ to } di}, 0, \dots, 0, \overbrace{-2(x^i - x^j)}^{d(j-1)+1\text{ to } dj}, 0, \dots, 0)\ .$$
Here every component $x^i - x^j$ is also a vector ($x^1,\dots, x^{k+1}\in \mathbb R^d$ are vectors). Thus we can see that $DF_{G_{k+1,m}}(\bm x)\cdot \bm u = 0$ is equivalent to (\ref{infrigcondition}).
\end{proof}

\begin{proposition}\label{polydep}
If $H\subseteq K_{k+1}$ is an independent (in $\mathbb R^d$) set of edges then $\bm x\in X_H$ if and only if the rows of $DF_{K_{k+1}}(\bm x)$ corresponding to edges of $H$ are linearly dependent. Moreover if $H \subset H'$ ($H,H'$ independent), then $X_H \subseteq X_{H'}$.
\end{proposition}
\begin{proof}
Since the rank of the matrix is less than $|H|$, the $|H|$ row vectors are linearly dependent. Conversely, if the rows are linearly dependent every minor has to be zero since all the submatrices will satisfy the same dependence.

Moreover, if $H\subset H'$ then the matrix corresponding to $H'$ will have rank at least that of the one for $H$.
\end{proof}
\begin{theorem}\label{genericthm}
The set of generic tuples in $\mathbb R^d$ is an open dense set of full Lebesgue measure. Moreover every independent (in $\mathbb R^d$) set $H$ is in fact independent in $\mathbb R^d$ with respect to any generic tuple in $\mathbb R^d$.
\end{theorem}
\begin{proof}
Note that to each polynomial $P_H$ corresponds at least one tuple $\bm x_0$ for which $P_H$ is nonzero, thus the zero sets $X_H$ are proper algebraic varieties. Thus the set of generic tuples of $G$ is nonempty, and in particular open dense of full measure (since the complement $X$ is of codimension at least $1$, as a proper algebraic variety).

The independence of $H$ for any generic tuple follows from the definition of genericity. In particular if $\bm x$ is generic then $\bm x\not \in X$, thus $\bm x \not \in X_H$. By Proposition \ref{polydep} it follows that $H$ is independent with respect to $\bm x$.
\end{proof}
\begin{lemma}\label{affinetrans}
If $A$ is an invertible affine transformation of $\mathbb R^d$ and we set $A\bm x = (Ax^1, \dots, Ax^{k+1})$, then we have that $AX = X$. Thus invertible affine transformations preserve genericity.
\end{lemma}
\begin{proof}
If $Ax = Bx + b$ where $B$ is an invertible linear transformation and $b$ a vector, then since the row vectors of $DF_{K_{k+1}}(\bm x)$ contain the entries $\pm (x^i - x^j)$, we see that the row vectors of $DF_{K_{k+1}}(A\bm x)$ contain the entries $\pm(Ax^i - Ax^j) = \pm B(x^i - x^j)$. In particular we see that the rows of $DF_{K_{k+1}}(\bm x)$ corresponding to $H$ are linearly independent if and only if those of $DF_{K_{k+1}}(\bm Ax)$ are. Using Proposition \ref{polydep} we conclude that affine transformations preserve genericity.
\end{proof}

\begin{definition}
The tuple $\bm x = (x^1, \dots, x^{k+1})$ is said to be in \emph{general position in $\mathbb R^d$} if for every set $J \subseteq \{1,\dots,k+1\}$ with $|J| \leq d+1$ we have that $\{x^j : j\in J\}$ is affinely independent. 
\end{definition}
\begin{theorem}Assume $q \leq d$. The edge set of $K_{q+1}$ is then independent in $\mathbb R^d$, in fact with respect to any tuple in general position.\label{completegraphindep}
\end{theorem}
\begin{proof}
Let $\bm x_0$ be in general position and assume that the rows of $DF_{K_{q+1}}(\bm x_0)$ are linearly dependent, say $$\sum_{1\leq i<j\leq q} s_{ij} r_{ij} = 0$$ where $r_{ij}$ is the row corresponding to the edge $ij\in K_{q+1}$. Suppose $a<b$ is such that $s_{ab} \not = 0$. Then if we focus on the column corresponding to $x^a$ we get the nontrivial equation
$$\sum_{j < a} s_{ja} (x^a - x^j) + \sum_{j > a} s_{aj} (x^a - x^j) = 0$$
which contradicts the fact that $\bm x_0$ is in general position, i.e. that $\{x^a-x^j : j \not = a\}$ is a linearly independent set of vectors. 
\end{proof}
\begin{theorem}\label{genericgeneral}
Let $\bm x$ be a generic tuple in $\mathbb R^d$. Then  $\bm x$ is in general position in $\mathbb R^d$.
\end{theorem}
\begin{proof}
Assume $\bm x$ is not in general position. Thus for some $1\leq q\leq d$ without loss of generality we may assume $\{x^1, \dots, x^{q+1}\}$ are affinely dependent with $\{x^1, \dots, x^q\}$ affinely independent. As proven in Theorem \ref{completegraphindep}, the edge set of $K_{q+1}$ is independent. Let $A$ be the affine transformation taking each $x^j$, $1\leq j \leq q$ to $e_j$, the standard $j$-th basis vector (using Lemma \ref{affinetrans}). Then by affine dependence we must have $Ax^{q+1} = t_1e_1 + \cdots + t_q e_q$ with $t_1 + \cdots + t_q = 0$. Setting $s_{ij} = -t_it_j$ for $1\leq i < j \leq q$ and $s_{i(q+1)} = t_i$ for $1\leq i \leq q$ we easily check that for any $1 \leq a \leq q+1$, 
$$\sum_{ia\in K_{q+1}}s_{ia} (x^a - x^i) + \sum_{aj\in K_{q+1}} s_{aj} (x^a - x^j) = 0\ .$$

Thus the edge set of $K_{q+1}$ is not independent with respect to $(x^1, \dots, x^{q+1})$ and so by Theorem \ref{genericthm} we conclude that $(x^1, \dots, x^{q+1})$ is not a generic tuple. We now note that if $\bm x$ were a generic tuple then $(x^1, \dots, x^{q+1})$ would also be a generic tuple, simply because we are removing $d(k+1) - d(q+1)$ column vectors from $DF_{K_{k+1}}$ to test for the genericity of $(x^1, \dots, x^{q+1})$. This contradicts what we have found therefore $\bm x$ is not generic.
\end{proof}
\begin{lemma}\label{confspacedim}
If $\bm x$ is in general position in $\mathbb R^d$ then $\dim \mathcal D(\bm x) = \binom{d+1}2$.
\end{lemma}
\begin{proof}
First let us show that if $\gamma(t) : (-1,1) \to M_{\bm x}$ is a smooth curve with $\gamma(0) = \bm x$, then $\gamma'(0) \in \mathcal D(\bm x)$. This follows from the fact that the composition $F(\gamma(t))$ is constant, so by the chain rule $DF(\gamma(0))\cdot \gamma'(0) = 0$. Note that $\dim M_{\bm x} = \binom{d+1}2$, giving us 
$$\dim \mathcal D(\bm x) \geq \binom{d+1}2\ .$$
For the reverse inequality, we will show that any infinitesimal motion $\bm u\in\mathcal D(\bm x)$ projects injectively to an infinitesimal motion $\bm{\tilde u}$ of $\mathcal V(K_{d+1}, \bm {\tilde x})$ where $\bm{\tilde u} = (u^1,\dots,u^{d+1})$ and $\bm{\tilde x} = (x^1, \dots, x^{d+1})$. By the rank-nullity theorem and Theorem \ref{completegraphindep},
$$\dim \mathcal V(K_{d+1}, \bm {\tilde x}) = d(d+1) - \binom{d+1}2 = \binom{d+1}2\ ,$$
establishing the reverse inequality and completing the proof.

Thus it remains to show injectivity. Let $\bm u, \bm v \in \mathcal D(\bm x)$ and assume $\bm {\tilde u} = \bm {\tilde v}$, that is, $u^i = v^i$ for $1 \leq i \leq d+1$. Since the space $\mathcal D(\bm x)$ is a vector space, for $\bm w = \bm u - \bm v$ we have
$$(w^i - w^j)\cdot(x^i - x^j) = 0, \text{ for } ij\in K_{k+1}\ .$$
Now if $i=1,\dots,d$ we have $w^i = 0$, so that for any $j > d$ we have
$$w^j \cdot (x^i - x^j) = 0, \text{ for } i = 1,\dots, d\ .$$
But that means $w^j$ is perpendicular to $d$ linearly independent vectors, since $\bm x$ is in general position, thus $w^j = 0$ as well, and $\bm u = \bm v$.  
\end{proof}
\begin{theorem}\label{onegenallgen}
Let $(G_{k+1,m}, \bm x_0)$ be an infinitesimally rigid framework in $\mathbb R^d$ with $\bm x_0$ generic. Then for all generic tuples $\bm x$, the frameworks $(G_{k+1,m}, \bm x)$ are infinitesimally rigid in $\mathbb R^d$.
\end{theorem}
\begin{proof}
Since $\mathcal V(G_{k+1,m}, \bm x_0) = \mathcal D(\bm x_0)$, combine Lemma \ref{confspacedim} and Theorem \ref{genericgeneral} to obtain $$\dim \mathcal V(G_{k+1,m}, \bm x_0) = \binom{d+1}2\ .$$
Since $\mathcal V(G_{k+1,m}, \bm x_0) = \ker DF_{G_{k+1,m}}(\bm x_0)$ by the rank-nullity theorem we obtain
\begin{align}\dim \mathcal V(G_{k+1,m}, \bm x_0) = d(k+1) - \operatorname{rank} DF_{G_{k+1,m}} (\bm x_0) \ .\label{rankdimeq}\end{align}
Combining these two equations we find that
$$\operatorname{rank} DF_{G_{k+1,m}}(\bm x_0) = d(k+1) - \binom{d+1}2\ .$$
Since generic tuples have the same rank (by Theorem \ref{genericthm}), by using Equation (\ref{rankdimeq}) with $\bm x$ in place of $\bm x_0$ we see that $\dim\mathcal V(G_{k+1,m},\bm x) = \binom{d+1}2$, implying that $\mathcal V(G_{k+1,m},\bm x) = \mathcal D(\bm x)$, so that $(G_{k+1,m}, \bm x)$ is infinitesimally rigid.
\end{proof}

\begin{corollary}\label{noofedges}
A minimally infinitesimally rigid graph $G_{k+1,m}$ in $\mathbb R^d$ satisfies
$$m = d(k+1)-\binom{d+1}2\ .$$
\end{corollary}
\begin{proof}
Let $G_{k+1,m}$ be minimally infinitesimally rigid. Then $m \geq \operatorname{rank} DF_{G_{k+1,m}}(\bm x_0) = d(k+1) - \binom{d+1}2$, as the proof of Theorem \ref{onegenallgen} shows. Let $(G_{k+1,m}, \bm x_0)$ be a regular framework. If we assume $m > d(k+1) - \binom{d+1}2$, there must be a subset $H$ of edges of $G_{k+1,m}$ such that $$\operatorname{rank} DF_H = d(k+1) - \binom{d+1}2 = |H|$$ about $\bm x_0$, therefore $H$ is infinitesimally rigid about $\bm x_0$. Since that is an open set, by Theorem \ref{onegenallgen} the generic behavior of $H$ is determined, thus $H$ is infinitesimally rigid, which is a contradiction since $H$ has less edges than $G_{k+1,m}$.
\end{proof}
\begin{proposition}\label{uniquereal}
Let $G_{k+1,m}$ be a minimally infinitesimally rigid graph in $\mathbb R^d$ and $(G_{k+1,m}, \bm x_0)$ be a regular framework. Then there exists some open neighborhood $U$ of $\bm x_0$ and an embedded $m$-dimensional submanifold $M\subset U$ that contains $\bm x_0$, with $F_{G_{k+1,m}}$ restricted on $M$ a diffeomorphism onto its image. Moreover if $\bm x\in U$, letting $N_{\bm x} = \{ \bm y : F_G(\bm y) = F_G(\bm x)\}$ denote the level curves, we have
$$N_{\bm x} \cap U = \{(Tx^1, \dots, Tx^{k+1}) : T\in \operatorname{ISO}(d)\}\cap U\ .$$
The same also hold for $f_{G_{k+1,m}}$.
\end{proposition}

\begin{center}
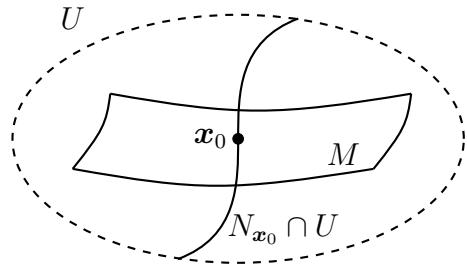
\begin{figure}[h]
\begin{tikzpicture}
\draw[thick] (-1,0) .. controls (.7,-.3) and (1.1,-.3) .. (3,0);
\draw[thick] (-.5,1) .. controls (1.2,.7) and (1.7,.7) .. (3.5,1);
\draw[thick] (-1,0) .. controls (-.6, .5) and (-.55, .6) .. (-.5, 1);
\draw[thick] (3,0) .. controls (3.4, .5) and (3.45, .6) .. (3.5, 1);
\draw[thick] (1.2,.4) .. controls (1.2, 1) and (1.2, 1.7) .. (2, 2);
\draw[thick] (1.2,.4) .. controls (1.2, -.2) and (1.2, -.9) .. (.4, -1.2);
%\draw[thick] (1.2, 0.6) -- (1.35, 0.55); 
%\draw[thick] (1.35, 0.55) -- (1.35, 0.35);
\filldraw (1.2, 0.4) circle (2pt);
\draw (1.2, .4) node[anchor=east] {$\bm x_0$};
\draw (2.6, .2) node {$M$};
\draw (1.8, -.8) node {$N_{\bm x_0}\cap U$};
\draw (-1, 2) node {$U$};
\draw[thick,dashed] (1.2,0.4) ellipse (3cm and 1.65cm);
\end{tikzpicture}
\caption{The local behavior of the distance function at a regular tuple.}\label{figure3}
\end{figure}
\end{center}

\begin{proof}

Since $\operatorname{rank} DF_{G_{k+1,m}}(\bm x_0) = m$ is maximal, it stays maximal around an open neighborhood $U$ of $\bm x_0$. The Inverse Function Theorem yields local coordinates $(\bm p, \bm q)$ at $\bm x_0$ such that $F_{G_{k+1,m}}(\bm p, \bm q) = \bm p$. The manifold $M$ is the image of $(\bm p, 0)$. It clearly is of dimension $m$. Using Corollary \ref{noofedges} and the fact that $\dim \operatorname{ISO}(d) = \binom{d+1}2$, the other claims follow.
\end{proof}
\begin{remark}We justly say that the regular frameworks of a minimally infinitesimally rigid graph are locally uniquely realizable, in the sense that modulo isometries the distance function is a local diffeomorphism as Proposition \ref{uniquereal} shows.\end{remark}
\subsection{Useful Lemmas}
Here we collect the rest of lemmas that were used, that are not related to graph rigidity.

\begin{lemma}\label{basic1}
Let $\phi\in C_c^\infty(\mathbb R^m)$ be a nonnegative radial function with $\int \phi = 1$, $\phi \leq 1$ and $\operatorname{supp} \phi \subset \{\bm t : |\bm t | \leq 2\}$ and for $\epsilon > 0$ set $\phi^\epsilon(\bm t) = \epsilon^{-m}\phi(\epsilon^{-1}\bm t)$. Let $\nu(\bm t)$ be a Borel probability measure. Then $\nu^\epsilon = \phi^\epsilon *\nu$ converges weakly-* to $\nu$.
\end{lemma}
\begin{proof}
Let $f\in C_0(\mathbb R^m)$ be a nonnegative function that vanishes at infinity. Then
\begin{align*}
\int f d\nu^\epsilon = \int (f * \phi^\epsilon) d\nu\ .
\end{align*}
It is a well known result of mollifiers that $f*\phi^\epsilon \to f$ pointwise and by the Dominated Convergence Theorem we may conclude that $\int f d\nu^\epsilon \to \int f d\nu$.
\end{proof}
\begin{lemma}\label{basic2}
Let $V$ be a normed vector space and $V^*$ its dual equiped with the operator norm. If $\nu_n \to \nu$ weakly-* in $V^*$ then $\liminf \|\nu_n\| \geq \|\nu\|$. 
\end{lemma}
\begin{proof}
\begin{align*}
\lim_{n\to\infty} \inf_{k\geq n} \sup_{\|x\|=1} \langle x, \nu_k\rangle \geq \lim_{n\to\infty}\sup_{\|x\|=1} \inf_{k\geq n} \langle x, \nu_k\rangle \geq \sup_{\|x\|=1} \langle x, \nu\rangle\ .
\end{align*}
\end{proof}
\begin{lemma}\label{frstman}
If $\mu$ is a measure on $\mathbb R^n$ and $C_\mu > 0$ a constant with $\mu(B(x,r)) \leq C_\mu r^s$ for some $s > 0$ and all $r>0$ and $x\in \mathbb R^n$, then for any measurable subset $A$ of $\mathbb R^n$ with $\mu(A) > 0$ we have $\dim A \geq s$.
\end{lemma}
\begin{proof}
Let $U_j$ be open balls of radius $r_j$ covering $A$. We then have $$0 < \mu(A) \leq \sum_{j=1}^\infty \mu(U_j) \leq C_\mu\sum_{j=1}^\infty r_j^s.$$
Taking the infimum over all such collections $U_j$ we obtain that $\mathcal H^s(A) > 0$.
\end{proof}

\begin{lemma}\label{sidedudelemma}
Let $\sigma_t$ denote the surface measure of the sphere $tS^{d-1} \subset \mathbb R^d$ of radius $t$ centered at $0$. Let $\phi^\epsilon(x) =  \epsilon^{-d}\phi(\epsilon^{-1}x)$ and $\phi\in C_c^\infty(\mathbb R^d)$ is a nonnegative radial function with $\int \phi = 1$, $\phi \leq 1$ and $\operatorname{supp}\phi \subset B(0, 2)$. Let $\sigma_t^\epsilon = \phi^\epsilon*\sigma_t$. Let $\mu$ be a Frostman measure on $E\subset\mathbb R^d$, $E$ compact, with Frostman exponent $s > \frac{d+1}2$. Then there exists a constant $C_t > 0$ independent of $\epsilon$ with
$$\|\sigma^\epsilon_t * \mu\|_{L^1(\mu)} < C_t\ .$$
\end{lemma}
\begin{proof}
We use Plancherel and the stationary phase of the sphere, (see \cite{W04}), that tells us that for $\xi$ of large norm and some $c>0$ we have
$$|\widehat{\sigma^\epsilon_t}(\xi)| \leq c t^{\frac {d-1}2} |\xi|^{-\frac {d-1}2}\ ,$$
to obtain that for some $C>0$ depending on the diameter of $E$,
\begin{align*}
\int \sigma^\epsilon_t * \mu (x) d\mu(x) & = \int \widehat{\sigma_t}(\xi)\widehat{\phi}(\epsilon\xi)|\widehat\mu|^2(\xi) d\xi \\
& \lesssim 1 + \int |\xi|^{-\frac{d-1}2} |\widehat\mu|^2(\xi) d\xi \\
& \lesssim 1 + \int\int |x - y|^{-\frac {d+1}2}d\mu(x)d\mu(y) \\
& \lesssim 1 + \int \int_C^{+\infty} \mu\left\{x : |x-y| < \lambda^{-\frac 2 {d+1}}\right\} d\lambda d\mu(y) \\
& \lesssim 1 + \int_C^{+\infty} \lambda^{-\frac 2 {d+1}s} d\lambda
\end{align*}
The last integral is finite by assumption.
\end{proof}

\vskip.125in 

\section{Proof of Theorem \ref{necessaryexponent}} 

\vskip.125in 

Let $q$ be a positive integer and define $E_q$ to be the $q^{-\frac{d}{s}}$-neighborhood of $\frac{1}{q} \left\{ {\Bbb Z}^d \cap {[0,q]}^d \right\} $ with $s \in \left( \frac{d}{2}, d \right)$ to be determined later. It is known (see e.g. \cite{Falc85}) that if we choose $q_1=2$, $q_{i+1}>q_i^i$, then the Hausdorff dimension of $E=\cap_i E_{q_i}$ is equal to $s$. 

\begin{lemma} \label{discretecongruencecount} The number of congruence classes of frameworks with $k+1$ vertices in ${\Bbb Z}^d \cap {[0,q]}^d$ is bounded above by $Cq^{dk}$. \end{lemma} 

To prove the lemma, fix one of the vertices at the origin, which we may do since ${\Bbb Z}^d$ is translation invariant. The number of the remaining $k$-tuples is $\leq q^{dk}$ by construction. This proves the lemma. 

We now consider an infinitesimally rigid framework on $k+1$ vertices in ${\Bbb Z}^d \cap {[0,q]}^d$ described by the graph $G_{k+1,m}$. By Corollary \ref{noofedges}, the number of edges is $m=d(k+1)-{d+1 \choose 2}=dk-{d \choose 2}$. It follows that 
$$ {\mathcal H}^{dk-{d \choose 2}}(\Delta(G_{k+1,m}, E_q^{k+1})) \leq C {(q^{-\frac{d}{s}})}^{dk-{d \choose 2}} \cdot 
q^{dk}.$$ 

This quantity tends to $0$ as $q \to \infty$ if $s<d-\frac{{d \choose 2}}{k}$ and Theorem \ref{necessaryexponent} is proved.

\end{document}